\documentclass[letterpaper,12pt,oneside,final]{article}
 
\newcommand{\href}[1]{#1} % does nothing, but defines the command so the print-optimized version will ignore \href tags (redefined by hyperref pkg).
\usepackage{amsmath,amssymb,amstext} % Lots of math symbols and environments
\usepackage[pdftex]{graphicx} % For including graphics N.B. pdftex graphics driver
\usepackage{amsmath,amssymb,amstext,amsthm,amsfonts}
\usepackage{dsfont}
\usepackage[pdftex]{graphicx}
\usepackage{caption}
\usepackage{color}% Include colors for document elements
\usepackage{dcolumn}% Align table columns on decimal point
\usepackage{bm}% bold math
\usepackage{float}
\usepackage{multirow}

\usepackage{algorithm} % For counting chapters
\usepackage{algorithmicx, algpseudocode}
\usepackage[pdftex,pagebackref=false]{hyperref} % with basic options
\definecolor{background-color}{gray}{0.98}
\definecolor{steelblue}{rgb}{0.27, 0.51, 0.71}
\definecolor{brickred}{rgb}{0.8, 0.25, 0.33}
\definecolor{bluegray}{rgb}{0.4, 0.6, 0.8}
\definecolor{amethyst}{rgb}{0.6, 0.4, 0.8}

\hypersetup{
    plainpages=false,       % needed if Roman numbers in frontpages
    unicode=false,          % non-Latin characters in Acrobat's bookmarks
    pdftoolbar=true,        % show Acrobats toolbar?
    pdfmenubar=true,        % show Acrobat's menu?
    pdffitwindow=false,     % window fit to page when opened
    pdfstartview={FitH},    % fits the width of the page to the window
    pdftitle={Toeplitz\ Eigenvalues},    % title: CHANGE THIS TEXT!
    pdfauthor={Chris Salahub},    % author: CHANGE THIS TEXT! and uncomment this line
    pdfsubject={Statistics},  % subject: CHANGE THIS TEXT! and uncomment this line
    pdfnewwindow=true,      % links in new window
    colorlinks=true,        % false: boxed links; true: colored links
    linkcolor=steelblue,         % color of internal links
    citecolor=brickred,        % color of links to bibliography
    filecolor=magenta,      % color of file links
    urlcolor=cyan           % color of external links
}

\algblock{Input}{EndInput}
\algnotext{EndInput}

\newtheorem{definition}{Definition}
\newtheorem{theorem}{Theorem}

\newcommand{\ve}[1]{\mathbf{#1}}           % for vectors
\newcommand{\sv}[1]{\boldsymbol{#1}}   % for greek letters
\newcommand{\m}[1]{\mathbf{#1}}               % for matrices
\newcommand{\sm}[1]{\boldsymbol{#1}}   % for greek letters
\newcommand{\tr}[1]{{#1}^{\mkern-1.5mu\mathsf{T}}}              % for transpose
\newcommand{\conj}[1]{{#1}^{\ast}}
\newcommand{\norm}[1]{||{#1}||}              % norm
\newcommand{\frob}[1]{\norm{#1}_F}
\newcommand{\abs}[1]{\lvert{#1}\rvert}              % norm
\newcommand*{\mvec}{\operatorname{vec}}
\newcommand*{\trace}{\operatorname{trace}}

\newcommand{\ind}[2]{I_{#2} \left( #1 \right)}
\DeclareMathOperator*{\argmin}{arg\,min}

\newcommand{\field}[1]{\mathbb{#1}}
\newcommand{\Reals}{\field{R}}

\newcommand{\Complex}{\field{C}}

\title{Approximating real symmetric Toeplitz matrices using the nearest circulant}
\author{Chris Salahub \\ {\footnotesize University of Waterloo, \texttt{csalahub@uwaterloo.ca}}}

\begin{document}

\maketitle

\begin{abstract}
  The nearest circulant approximation of the real Toeplitz matrix $\sm{\Sigma}$ in the Frobenius norm is derived. This matrix, $\m{C}_{\Sigma}$, is symmetric. It is proven that symmetric circulant matrices are the only real circulant matrices with all real eigenvalues. The Frobenius norm of the difference $\sm{\Sigma} - \m{C}_{\Sigma}$ for the case of $\sm{\Sigma}$ displaying exponential decay is evaluated using an expression of $\sum_{n=1}^{N} n^k p^n$ in terms of the first $k$ geometric moments. Compared to a classic approximation based on \cite{grenanderszego1958}, $\m{C}_{\Sigma}$ displays dramatically better behaviour in many finite cases, though both share the same leading term for large $M$.
\end{abstract}

\section{Introduction}

A number of methods exist to approximate the symmetric Toeplitz matrix
\begin{equation} \label{eq:multipleTesting:genEigCov}
  \sm{\Sigma} = \begin{bmatrix}
    \rho_0 & \rho_1 & \rho_2 & \dots & \rho_{M-1} \\
    \rho_1 & \rho_0 & \rho_1 & \dots & \rho_{M-2} \\
    \rho_2 & \rho_1 & \rho_0 & \dots & \rho_{M-3} \\
    \vdots & \vdots & \vdots & \ddots & \vdots \\
    \rho_{M-1} & \rho_{M-2} & \rho_{M-3} & \dots & \rho_0
  \end{bmatrix},
\end{equation}
where $\rho_0, \dots, \rho_{M-1} \in \Reals$. $\sm{\Sigma}$ appears regularly in statistical applications such as the correlation matrix of the discrete autoregressive process of order one, in information theory in certain filtering tasks \cite{gray2006toeplitz}, and in genetics in the correlation between disjoint measured sequences \cite{salahub2022correlation}. In the case of the equidistant genetic survey of \cite{LanderBotstein1989} and the autoregressive time series, the matrix has the more specific form
\begin{equation} \label{eq:multipleTesting:specEigCov}
  \sm{\Sigma}_e = \begin{bmatrix}
    1 & \rho & \rho^2 & \dots & \rho^{M-1} \\
    \rho & 1 & \rho & \dots & \rho^{M-2} \\
    \rho^2 & \rho & 1 & \dots & \rho^{M-3} \\
    \vdots & \vdots & \vdots & \ddots & \vdots \\
    \rho^{M-1} & \rho^{M-2} & \rho^{M-3} & \dots & 1
  \end{bmatrix}
\end{equation}
where $\rho \in [0, 1]$ is a real constant.

Of interest in many applications are the eigenvalues and eigenvectors of $\sm{\Sigma}$, call them $\ve{V}_1, \ve{V}_2, \dots, \ve{V}_M$ and $\lambda_1 \geq \lambda_2 \geq \dots \geq \lambda_M$ respectively. For brevity, the combination of these vectors and values is here referred to as the \emph{eigensystem} of $\sm{\Sigma}$. \cite{gray2006toeplitz} demonstrates the application of this particular eigensystem in signal processing, while \cite{cheverud2001}, \cite{LiJi2005}, and \cite{Galwey2009} use the eigenvalues to adjust for dependent multiple tests in the genomic context.

The typical approach to obtain is eigensystem is to approximate $\sm{\Sigma}$ with a matrix with a solved eigensystem. \cite{cheverud2001}, \cite{LiJi2005}, and \cite{Galwey2009} present different approximations based on the matrix
\begin{equation} \label{eq:multipleTesting:commonCor}
  \m{A}(\rho) = \rho \ve{1} \tr{\ve{1}} + (1 - \rho) \m{I}_M.
\end{equation}
$\m{A}(\rho)$ has a known eigensystem for any $\rho$ and matches the eigensystem of $\sm{\Sigma}$ for the edge cases where $\rho_1 = \rho_2 = \dots = 0$ or $\rho_1 = \rho_2 = \dots = 0$ and $\rho_0 = 1$.

\cite{gray2006toeplitz} and \cite{grenanderszego1958} instead utilize circulant matrices to approximate the eigensystem of $\sm{\Sigma}$. Asymptotically, $\sm{\Sigma}$ has the same eigensystem as certain circulant matrices and the eigensystem of any circulant matrix is known.

Absent from the asymptotic approach of \cite{gray2006toeplitz} or the myriad approaches of \cite{cheverud2001}, \cite{LiJi2005}, and \cite{Galwey2009} is a consideration of the finite, non-edge cases of $\sm{\Sigma}$ and how it might be best approximated. This work derives $\m{C}_{\Sigma}$, the nearest circulant matrix to $\sm{\Sigma}$ in the weak, or Frobenius, norm. $\m{C}_{\Sigma}$ is symmetric circulant, and so has only real eigenvalues. Indeed, it is proven that a real circulant matrix has real eigenvalues if and only if it is a symmetric circulant. $\m{C}_{\Sigma_e}$ is additionally shown to be asymptotically equivalent to $\sm{\Sigma}_e$ by considering the limit of the remainder $\sm{\Sigma_e} - \m{C}_{\Sigma_e}$. The approximation of $\sm{C}_{\Sigma}$ and the classic approximation of \cite{gray2006toeplitz} are compared, with the nearest circulant provided a much better approximation in the cases of $\rho$ close to 1 and small $M$.

Section \ref{c:multipleTesting:circDecom} presents an introduction to circulant matrices and the proof that the only real circulants with all real eigenvalues are symmetric. Section \ref{c:multipleTesting:nearestCirc} derives $\m{C}_{\Sigma}$ and $\m{C}_{\Sigma}$ and states their eigensystems. The asymptotic equivalence of $\sm{\Sigma}_e$ and $\m{C}_{\Sigma}$ are then present in Section \ref{c:multipleTesting:asympEquiv} before a comparison of the approximations is present in Section \ref{c:multipleTesting:rateConverge}.

\section{Circulant matrices} \label{c:multipleTesting:circDecom}

For clarity working with different matrices, let the function $\lambda_k(\m{C})$ return the $k^{\text{th}}$ eigenvalue of $\m{C}$, which may not be ordered by magnitude.

A complex matrix $\m{C} \in \Complex^{M \times M}$ is called circulant if the $i^{\text{th}}$ row is given by the cyclic shift $i$ elements rightward of a vector of $M$ elements, typically denoted $(c_0, c_1, c_2, \dots, c_{M-1})$. Explicitly
\begin{equation} \label{eq:explicitCirculant}
  \m{C} = \begin{bmatrix}
    c_0 & c_1 & c_2 & \dots & c_{M-2} & c_{M-1} \\
    c_{M-1} & c_0 & c_1 & \dots & c_{M-3} & c_{M-2} \\
    c_{M-2} & c_{M-1} & c_0 & \dots & c_{M-4} & c_{M-3} \\
    \vdots & \vdots & \vdots & \ddots & \vdots & \vdots \\
    c_2 & c_3 & c_4 & \dots & c_0 & c_1 \\
    c_1 & c_2 & c_3 & \dots & c_{M-1} & c_0 \\
    \end{bmatrix}
\end{equation}
So every circulant matrix $\m{C}$ can be specified by its first row alone. Moreover, this first row corresponds to the coefficients in a convenient expression of $\m{C}$ as a matrix polynomial. Let $\m{P}$ be the circulant matrix with $c_0 = c_2 = c_3 = \dots = c_{M-1} = 0$ and $c_1 = 1$. That is,
\begin{equation} \label{eq:pDef}
  \m{P} = [ \ve{e}_M | \ve{e}_1 | \dots | \ve{e}_{M-1} ]
\end{equation}
where $\ve{e}_i$ is the $i^{\text{th}}$ basis vector. Then $\m{P}$ is the permutation matrix corresponding to a cyclic shift of all elements of a vector $\ve{x} \in \Complex^M$ one to the right. Due to this cyclic shift property, it is also straightforward to note that
\begin{equation} \label{eq:powerPDef}
  \m{P}^m = \m{P} \m{P} \dots \m{P} = [ \ve{e}_{M-m+1} | \ve{e}_{M-m+2} | \dots | \ve{e}_M | \ve{e}_1 | \dots | \ve{e}_{M-m} ].
\end{equation}
Using these $\m{P}^m$, $\m{C}$ can be written
\begin{equation} \label{eq:circMatPol}
  \m{C} = c_0 \m{I}_M + c_1 \m{P} + c_2 \m{P}^2 + \dots + c_{M-1} \m{P}^{M-1},
\end{equation}
from which the eigensystem of $\m{C}$ can be derived from the eigensystem of $\m{P}$. Using a cofactor expansion of $\det (\m{P} - \lambda \m{I})$ it can be shown that the eigenvalues of $\m{P}$ are the $M^{\text{th}}$ roots of unity, that is
$$\lambda_k (\m{P}) = \left ( e^{\frac{2 \pi i}{M}} \right )^k = \omega^k$$
where $k = 0, \dots, M-1$. The corresponding eigenvectors $\ve{x}_k$ are then
\begin{equation} \label{eq:multipleTesting:circEigenVec}
  \ve{x}_k = \begin{bmatrix}
  1 \\
  \omega^k \\
  \omega^{2k} \\
  \vdots \\
  \omega^{(M-1)k}
\end{bmatrix},
\end{equation}
which can be seen by considering $\m{P} \ve{x}_k = \lambda_k(\m{P}) \ve{x}_k$. Note that any eigenvector of $\m{P}$ with eigenvalue $\lambda$ is also an eigenvector of $\m{P}^m$ with eigenvalue $\lambda^m$, and so the eigenvalues of $\m{C}$ are given by
\begin{equation} \label{eq:multipleTesting:circEigenVals}
  \lambda_k (\m{C}) = c_0 + \sum_{m = 1}^{M-1} c_m \omega^{mk}
\end{equation}
with corresponding eigenvectors $\ve{x}_k$ as above for $k = 0, \dots, M-1$. A particular circulant matrix structure is of interest.

\begin{definition}[Symmetric circulant] \label{def:symmCirc}
  A circulant matrix $\m{C} \in \Complex^{M \times M}$ is symmetric if elements in its first row $c_0, c_1, c_2, \dots, c_{M-1}$ satisfy $c_m = c_{M-m}$ for all $m \geq 1$.
\end{definition}

To provide a visual example of such a matrix, consider the circulant with $c_m = \min \{m, M-m\}$. When $M = 6$ we have
$$\begin{bmatrix}
  0 & 1 & 2 & 3 & 2 & 1 \\
  1 & 0 & 1 & 2 & 3 & 2 \\
  2 & 1 & 0 & 1 & 2 & 3 \\
  3 & 2 & 1 & 0 & 1 & 2 \\
  2 & 3 & 2 & 1 & 0 & 1 \\
  1 & 2 & 3 & 2 & 1 & 0
\end{bmatrix}$$
This circulant is symmetric, and any symmetric circulant matrix will have an analogous structure. By basic results of linear algebra, it follows that any symmetric circulant matrix will have only real eigenvalues. Indeed, a circulant matrix will have only real eigenvalues for any $M$ if and only if it is symmetric.

\begin{theorem}[Real eigenvalues of symmetric circulants] \label{thm:symmEigen}
  A circulant matrix $\m{C} \in \Reals^{M \times M}$ has real eigenvalues if and only if it is a symmetric circulant matrix.
\end{theorem}
\begin{proof}
  Consider the eigensystem of $\m{C}$. As it is circulant, it has eigenvalues
  $$\lambda_k (\m{C}) = c_0 + \sum_{m = 1}^{M-1} c_m \omega^{mk},$$
  or rather
  $$\lambda_k (\m{C}) = c_0 + \sum_{m = 1}^{M-1} c_m \left ( \cos \frac{2\pi mk}{M} + i \sin \frac{2\pi mk}{M} \right )$$
  We can rewrite this to emphasize the real and imaginary components as
  $$\lambda_k (\m{C}) = \left ( \sum_{m = 0}^{M-1} c_m \cos \frac{2\pi mk}{M} \right ) + i \left ( \sum_{m = 1}^{M-1} c_m \sin \frac{2\pi mk}{M} \right ).$$
  If $\lambda_k (\m{C}) \in \Reals$ for all $k \in \{0, 1, \dots, M-1\}$, we must have
  $$\sum_{m = 1}^{M-1} c_m i \sin \frac{2\pi mk}{M} = 0 \hspace{0.5cm} \forall k \in \{0, 1, \dots, M-1\}.$$
  But note that
  $$i \sin \frac{2\pi mk}{M} = \frac{1}{2} \left ( e^{\frac{2\pi mk}{M} i} - e^{-\frac{2\pi mk}{M} i} \right )$$
  and
  $$e^{-\frac{2\pi mk}{M} i} = e^{-\frac{2\pi mk}{M} i + 2\pi ki} = e^{\frac{2\pi (M - m)k}{M} i},$$
  and so we require
  $$\sum_{m = 1}^{M-1} c_m \frac{1}{2} \left ( e^{\frac{2\pi mk}{M}i} - e^{\frac{2\pi (M-m)k}{M}i} \right ) = 0 \hspace{0.5cm} \forall k \in \{0, 1, 2, \dots, M-1\}.$$
  However,
  $$\sum_{m = 1}^{M-1} c_m \frac{1}{2} \left ( e^{\frac{2\pi mk}{M}i} - e^{\frac{2\pi (M-m)k}{M}i} \right ) = 0$$
  \\
  $$\iff \sum_{m = 1}^{M-1} c_m e^{\frac{2\pi mk}{M}i} - \sum_{m = 1}^{M-1} c_m e^{\frac{2\pi (M-m)k}{M}i} = 0$$
  \\
  $$\iff \sum_{m = 1}^{M-1} c_m e^{\frac{2\pi mk}{M}i} - \sum_{m = 1}^{M-1} c_{M-m} e^{\frac{2\pi mk}{M}i} = 0$$
  \\
  $$\iff \sum_{m = 1}^{M-1} ( c_m - c_{M-m} ) e^{\frac{2\pi mk}{M}i} = \sum_{m = 1}^{M-1} ( c_m - c_{M-m} ) \omega^{mk} = 0$$
  for all $k \in \{0, 1, \dots, M-1\}$. In other words, the eigenvalues of $\m{C}$ are all real if and only if the vector of differences $(c_m - c_{M-m})_{m=1,\dots,M-1}$ is a vector in the null space of
  $$\begin{bmatrix}
  1 & 1 & 1 & \dots & 1 \\
  \omega & \omega^2 & \omega^3 & \dots & \omega^{M-1} \\
  \omega^2 & \omega^4 & \omega^6 & \dots & \omega^{2(M-1)} \\
  \omega^3 & \omega^6 & \omega^9 & \dots & \omega^{3(M-1)} \\
  \vdots & \vdots & \vdots & \ddots & \vdots \\
  \omega^{M-1} & \omega^{2(M-1)} & \omega^{3(M-1)} & \dots & \omega^{(M-1)^2} \\
  \end{bmatrix}.$$
  Noting that these columns are eigenvectors of $\m{C}$, it can quickly be recognized that the null space of this matrix is the line $x_1 = x_2 = \dots = x_M$, as this is the final orthogonal eigenvector of $\m{C}$. Therefore, our differences must satisfy
  $$c_m - c_{M-m} = a$$
  for some $a \in \Complex$ for all $m \in \{1, 2, \dots, M-1\}$. If $M$ is even, then $M/2 \in \{1, 2, \dots, M-1\}$, and so when $m = M/2$ we get the difference $c_{M/2} - c_{M - M/2} = c_{M/2} - c_{M/2} = 0$. Hence, $a = 0$ is the only solution. If $M$ is odd, then we have in particular $c_{\lfloor M/2 \rfloor} - c_{M - \lfloor M/2 \rfloor} = c_{\lfloor M/2 \rfloor} - c_{\lfloor M/2 \rfloor + 1} = a = c_{\lfloor M/2 \rfloor + 1} - c_{\lfloor M/2 \rfloor} = c_{\lfloor M/2 \rfloor + 1} - c_{M - \lfloor M/2 \rfloor - 1}$, which is true only if $c_{\lfloor M/2 \rfloor + 1} = c_{\lfloor M/2 \rfloor}$ and $a = 0$. Therefore, real eigenvalues are ensured if and only if
  $$c_m - c_{M-m} = 0 \iff c_m = c_{M-m},$$
  the definition of a symmetric circulant.
\end{proof}

\section{The nearest circulant to $\m{\Sigma}$} \label{c:multipleTesting:nearestCirc}

Now consider the decomposition
\begin{equation} \label{eq:circDecomp}
  \sm{\Sigma} = \m{C} + \m{R}
\end{equation}
where $\m{C}$ is a circulant matrix and $\m{R} = \sm{\Sigma} - \m{C}$ is a matrix of the element-wise residuals between $\m{C}$ and $\sm{\Sigma}$. This reframing moves the discussion from the space of asymptotic results to the features of $\m{C}$ and $\m{R}$, a familiar framework for anyone accustomed to considering the residuals of a given approximation.

Of immediate and obvious interest is the closest circulant matrix to $\sm{\Sigma}$, call it $\m{C}_{\Sigma}$. Consider using the weak, or Frobenius, norm on matrices. First, introduce the \textit{vectorization} operator on a matrix $\m{A} \in \Complex^{M \times N}$, denoted $\mvec{\m{A}}$. This operator takes $\m{A} \in \Complex^{M \times N}$ and converts it to a vector in $\Complex^{MN}$ by appending columns in order. So, for example,
$$\mvec{\begin{bmatrix}
    1 & 2  \\
    4 & 5  \\
  \end{bmatrix}} =
\begin{bmatrix}
  1 \\ 4 \\ 2 \\ 5
  \end{bmatrix}.$$
The Frobenius norm is then given by
$$\frob{\m{A}} = \sqrt{\conj{(\mvec{\m{A}})} (\mvec{\m{A}})}$$
or equivalently
$$\frob{\m{A}} = \sqrt{\trace \left ( \conj{\m{A}} \m{A} \right )}$$
for a matrix $\m{A} \in \Complex^{M \times M}$ with complex conjugate $\conj{\m{A}}$. In the real case, this is
$$\frob{\m{A}} = \sqrt{\trace \left ( \tr{\m{A}} \m{A} \right )}$$
Let $\m{C}$ be an $M \times M$ circulant matrix with first row $( c_0, c_1, c_2, \dots, c_{M-1} ) \in \Reals^M$. Then by definition $\m{C}_{\Sigma}$ is given by $\argmin_{\m{C}} \frob{\sm{\Sigma} - \m{C}}$, or equivalently $\argmin_{\m{C}} \frob{\sm{\Sigma} - \m{C}}^2$. Taking the second of these, we have
\begin{equation} \label{eq:argMinDef}
  \m{C}_{\Sigma} = \argmin_{\m{C}} \frob{\sm{\Sigma} - \m{C}}^2.
\end{equation}
Considering that
\begin{equation*}
  \begin{split}
    \frob{\sm{\Sigma} - \m{C}}^2 & = \trace \left ( \tr{\left [ \sm{\Sigma} - \m{C} \right ]} \left [ \sm{\Sigma} - \m{C} \right ] \right ) \\
    & \\
    & = \left ( \trace \tr{\sm{\Sigma}} \sm{\Sigma} - \trace \tr{\sm{\Sigma}} \m{C} - \trace \tr{\m{C}} \sm{\Sigma} + \trace \tr{\m{C}} \m{C} \right )
  \end{split}
\end{equation*}
and $\trace \tr{\sm{\Sigma}} \sm{\Sigma}$ is constant in $\m{C}$, we need only consider minimizing
\begin{equation} \label{eq:argMinState}
  F(\m{C}) = \trace \tr{\m{C}} \m{C} - \trace \tr{\sm{\Sigma}} \m{C} - \trace \tr{\m{C}} \sm{\Sigma}.
\end{equation}
The first term of Equation (\ref{eq:argMinState}) is straightforward to express in terms of the $c_m$. As $\m{C}$ is circulant,
$$\trace \tr{\m{C}} \m{C} = M \sum_{m = 0}^{M-1} c_m^2.$$
The other terms can be evaluated by expressing $\m{C}$ as a matrix polynomial.

Equation (\ref{eq:circMatPol}) and the symmetry of $\sm{\Sigma}$ allow us to write
$$\tr{\sm{\Sigma}} \m{C} = \sm{\Sigma} \left ( c_0 \m{I} + \sum_{m = 1}^{M-1} c_m \m{P}^m \right ) = c_0 \sm{\Sigma} + \sum_{m = 1}^{M-1} c_m \sm{\Sigma} \m{P}^m$$
and similarly
$$\tr{\m{C}} \sm{\Sigma} = c_0 \sm{\Sigma} + \sum_{m = 1}^{M-1} c_m \tr{\left ( \m{P}^m \right )} \sm{\Sigma} =  c_0 \sm{\Sigma} + \sum_{m = 1}^{M-1} c_m \tr{\left (\sm{\Sigma} \m{P}^m \right )}.$$
Next, consider
\begin{equation*}
    \trace \sm{\Sigma} \m{P}^m  = \trace \left ( \sm{\Sigma} [\ve{e}_{M-m+1} | \ve{e}_{M-m+2} | \dots | \ve{e}_{M-m} ] \right ),
\end{equation*}
which can be evaluated by considering the $k^{\text{th}}$ row of $\sm{\Sigma}$, $\sv{\sigma}_k$. The first $k$ elements of this row are the descending sequence $\rho_{k-1}, \rho_{k-2}, \dots, \rho_0$, and the remaining elements are the ascending sequence $\rho_1, \rho_2, \dots, \rho_{M-k}$. Noting that the trace of a product of two matrices is simply a sum of the inner products of the rows of the first with the columns of the second, we obtain
\begin{equation} \label{eq:productTrace}
  \begin{split}
    \trace \sm{\Sigma} \m{P}^m & = \sum_{k = 1}^{m} \tr{\sv{\sigma}}_k \ve{e}_{M-m+k} + \sum_{k = 1}^{M - m} \tr{\sv{\sigma}}_{m+k} \ve{e}_{k} \\
    & \\
    & = \sum_{k=1}^m \rho_{M-m} + \sum_{k=1}^{M-m} \rho_m \\
    & \\
    & = m \rho_{M-m} + (M-m) \rho_m.
  \end{split}
\end{equation}

Equation (\ref{eq:productTrace}) can then be substituted into Equation (\ref{eq:argMinState}) using the decomposition of Equation (\ref{eq:circMatPol}) to give
\begin{equation} \label{eq:finalDiff}
  F(\m{C}) = M \sum_{m=0}^{M-1} c_m^2 - 2 \left ( M c_0 \rho_0 + \sum_{m=1}^{M-1} c_m (m \rho_{M-m} + (M - m) \rho_m) \right ).
\end{equation}
As $\argmin_{\m{C}} F(\m{C})$ is the same as $\argmin_{\m{C}} \frob{\sm{\Sigma} - \m{C}}$, we can now consider the values which minimize Equation (\ref{eq:finalDiff}) in order to find the nearest circulant matrix to $\sm{\Sigma}$. Taking
\begin{equation*} 
  \frac{\partial}{\partial c_m} F(\m{C}) = \begin{cases}
    2Mc_0 - 2M\rho_0 & \text{for } m = 0, \\
    & \\
    2Mc_m - 2 \left ( m \rho_{M-m} + (M - m) \rho_m \right ) & \text{otherwise},
  \end{cases}
\end{equation*}
and noting that the Hessian matrix is $2M\m{I}$ and hence is positive definite so any solutions to $\argmin F(\m{C})$ must be minima, we obtain
\begin{equation*}
  c_m = \begin{cases}
    \rho_0 & \text{for } m = 0, \\
    & \\
    \frac{m}{M} \rho_{M-m} + \frac{M-m}{M} \rho_m & \text{otherwise},
  \end{cases}
\end{equation*}
which can be re-expressed as
\begin{equation} \label{eq:closestCirc}
  c_m = \begin{cases}
    \rho_0 & \text{for } m = 0, \\
    & \\
    \rho_m + \frac{m}{M}(\rho_{M-m} - \rho_m) & \text{otherwise},
  \end{cases}
\end{equation}
to make the relationship between $c_m$ and $\rho_m$ clearer. An important consequence of this system of equations is that $c_m = \frac{m}{M} \rho_{M-m} + \frac{M-m}{M} \rho_m = \frac{M - (M - m)}{M} \rho_{M-m} + \frac{M - m}{M} \rho_{M - (M - m)} = c_{M-m}$, and so $\m{C}_{\Sigma}$ is a symmetric circulant matrix with $c_m$ defined as in Equation (\ref{eq:closestCirc}). Therefore, this nearest circulant will have only real eigenvalues. This is desirable for the multiple testing adjustment case of \cite{cheverud2001}, \cite{LiJi2005}, and \cite{Galwey2009}, where the effective number of tests is a real number that is some function of the eigenvalues.

So the optimal decomposition in the Frobenius norm is
\begin{equation} \label{eq:sigmaCirculantDecomp}
  \sm{\Sigma} = \m{C}_{\Sigma} + \m{R}_{\Sigma}
\end{equation}
where $\m{C}_{\Sigma}$ is the circulant matrix defined by Equation (\ref{eq:closestCirc}), that is
\begin{equation*}
  c_m = \begin{cases}
    \rho_0 & \text{for } m = 0, \\
    & \\
    \rho_m + \frac{m}{M}(\rho_{M-m} - \rho_m) & \text{otherwise},
  \end{cases}
\end{equation*}
and $\m{R}_{\Sigma} = \sm{\Sigma} - \m{C}_{\Sigma}$, and so the value in the $m^{\text{th}}$ off-diagonal of $\m{R}_{\Sigma}$ is $\rho_m - c_m = \rho_m - (\rho_m + \frac{m}{M}(\rho_{M-m} - \rho_m) = \frac{m}{M} (\rho_m - \rho_{M-m}).$

The eigenvalues of $\m{C}_{\Sigma}$ are given by a substitution of Equation (\ref{eq:closestCirc}) into Equation (\ref{eq:multipleTesting:circEigenVals}), giving
\begin{equation} \label{eq:multipleTesting:circApproxEig}
  \lambda_k (\m{C}_{\Sigma}) = \rho_0 + 2 \sum_{m = 1}^{M-1} \frac{M - m}{M} \rho_m \cos \frac{2 \pi mk}{M}.
\end{equation}
The corresponding eigenvectors are given by Equation (\ref{eq:multipleTesting:circEigenVec}).

Consider $\m{R}_{\Sigma}$ briefly. For large $M$ and small $m$, $\frac{m}{M} (\rho_{M-m} - \rho_m) \approx 0$ while when $m$ is close to $M$, $\frac{m}{M} (\rho_m - \rho_{M-m}) \approx \rho_m - \rho_{M-m}$. This implies that in the case of large $M$, $\m{R}_{\Sigma}$ will have vanishingly small values for the central off-diagonals, and values in the corners of approximately $\rho_m - \rho_{M-m}$.

In the particular case of $\sm{\Sigma}_e$ from Equation (\ref{eq:multipleTesting:specEigCov}), $\rho_m = \rho^m$. In this case, $\m{C}_{\Sigma_e}$ has entries
\begin{equation*}
  c_m =  \rho^m + \frac{m}{M}(\rho^{M-m} - \rho^m)
\end{equation*}
and $\m{R}_{\Sigma_e}$ is $\frac{m}{M} (\rho^m - \rho^{M-m})$ for the $m^{\text{th}}$ off-diagonal where $m \in \{0, 1, 2, \dots, M-1\}$. The eigenvalues of $\m{C}_{\Sigma_e}$ are therefore
\begin{equation*}
  \lambda_k (\m{C}_{\Sigma_e}) = 1 + 2 \sum_{m = 1}^{M-1} \frac{M - m}{M} \rho^m \cos \frac{2 \pi mk}{M}.
\end{equation*}

\section{Asymptotic equivalence of $\m{C}_{\Sigma_e}$ and $\m{\Sigma}_e$} \label{c:multipleTesting:asympEquiv}

The approximation matrix $\m{C}_{\Sigma_e}$ has been derived here based purely on minimizing $\frob{\sm{\Sigma}_e - \m{C}_{\Sigma_e}}$ without any of the asymptotic guarantees of \cite{grenanderszego1958}. \cite{gray2006toeplitz} derives a similar result by considering the asymptotic equivalence of matrices in the weak norm. Matrices $\m{A}$ and $\m{B}$ in $\Complex^{M \times M}$ are said to be asymptotically equivalent in the weak norm if
$$\lim_{M \rightarrow \infty} \frac{1}{\sqrt{M}} \frob{A - B} = 0.$$
Therefore, a natural consideration is the difference
\begin{equation} \label{eq:multipleTesting:asympEq}
  \lim_{M \rightarrow \infty} \frac{1}{\sqrt{M}} \frob{\sm{\Sigma}_e - \m{C}_{\Sigma_e}} = \lim_{M \rightarrow \infty} \frac{1}{\sqrt{M}} \frob{\m{R}_{\Sigma_e}}.
\end{equation}
Before taking the limit, note
$$\frac{1}{\sqrt{M}} \frob{\m{R}_{\Sigma_e}} =  \sqrt{ \frac{1}{M} \trace{\tr{\m{R}_{\Sigma_e}} \m{R}_{\Sigma_e}} },$$
which has the square
\begin{eqnarray}
    \frac{1}{M} \frob{\m{R}_{\Sigma_e}}^2 & = &  \frac{1}{M} \trace{\tr{\m{R}_{\Sigma_e}} \m{R}_{\Sigma_e}} \nonumber \\
    & & \nonumber \\
    & = & \frac{1}{M} \sum_{i = 0}^{M-1} \sum_{j = 0}^{M-1} \frac{\abs{i-j}^2}{M^2} \left ( \rho^{2\abs{i-j}} - 2 \rho^{\abs{i-j} + M - \abs{i-j}} + \rho^{2(M - \abs{i-j})} \right ) \nonumber \\
    & & \nonumber \\
    & = & \frac{2}{M^3} \sum_{m = 1}^{M-1} (M-m) m^2 \left ( \rho^{2m} - 2 \rho^{M} + \rho^{2(M - m)} \right ) \label{eq:asymTrace}
\end{eqnarray}
Evaluating this expression for $\rho < 1$ is made easier by considering the general sum
$$\sum_{m = 1}^{M} m^k \rho^m.$$
\begin{theorem}[Truncated geometric power series] \label{thm:trunkmoment}
  The finite sum
  $$\sum_{n = 1}^{N} n^k p^n$$
  with $\abs{p} < 1$ can be expressed as
  $$\frac{p}{1 - p} \left [ (1 - p^N) G^{(k)}(0) - p^{N} \sum_{l = 1}^k {k \choose l} N^l G^{(k-l)}(0) \right ]$$
  where
  $$G^{(k)}(0) = \frac{d^k}{dt^k} E[e^{tX}] \Big |_{t = 0}$$
  is the $k^{\text{th}}$ derivative of $G(t) = E[e^{tX}]$ evaluated at zero for $X \sim Geo(1 - p)$.
\end{theorem}
\begin{proof}
Multiplying by $\frac{1 - p}{1 - p}$, we obtain
$$\sum_{n = 1}^{N} n^k p^n = \frac{p}{1 - p} \sum_{n = 1}^{N} n^k p^{n-1} (1 - p).$$
Now
\begin{eqnarray}
    \sum_{n = 1}^{N} n^k p^{n-1} (1 - p) & = & \sum_{n = 1}^{\infty} n^k p^{n-1} (1 - p) - p^{N} \sum_{n = N + 1}^{\infty} n^k p^{n - N - 1} (1 - p) \nonumber \\
    & & \nonumber \\
    & = & \sum_{n = 1}^{\infty} n^k p^{n-1} (1 - p) - p^{N} \sum_{n = 1}^{\infty} (n + N)^k p^{n - 1} (1 - p) \nonumber \\
    & & \nonumber \\
    & = & \sum_{n = 1}^{\infty} n^k p^{n-1} (1 - p) - p^{N} \sum_{n = 1}^{\infty} \sum_{l = 0}^k {k \choose l} n^{k-l} N^l p^{n - 1} (1 - p) \nonumber \\
    & & \nonumber \\
    & = & (1 - p^N) \sum_{n = 1}^{\infty} n^k p^{n-1} (1 - p) \nonumber \\
    & & \hspace{0.5cm}- p^{N} \sum_{l = 1}^k {k \choose l} N^l \sum_{m = 1}^{\infty}  n^{k-l} p^{n - 1} (1 - p),  \label{eq:decomp1}
\end{eqnarray}
but
$$\sum_{n = 1}^{\infty} n^k p^{n-1} (1 - p)$$
is just the $k^{\text{th}}$ moment of a geometric distribution with a probability of sucess of $1 - p$. This distribution has a moment-generating function
\begin{equation} \label{eq:geoMGF}
G(t) = \frac{1 - p}{1 - e^t p}
\end{equation}
which can be used to evaluate the $k^{\text{th}}$ moment by taking the $k^{\text{th}}$ derivative and evaluating it at $t = 0$, denoted
\begin{equation*} 
G^{(k)}(0)
\end{equation*}
and substituted into Equation \ref{eq:decomp1} to give the more succinct
\begin{equation} \label{eq:sumByMoments}
  \sum_{n = 1}^{N} n^k p^{n-1} (1 - p) = (1 - p^N) G^{(k)}(0) - p^{N} \sum_{l = 1}^k {k \choose l} N^l G^{(k-l)}(0).
\end{equation}
This gives the result for the original sum
\begin{equation} \label{eq:kfirstmoments}
 \sum_{n = 1}^{N} n^k p^n = \frac{p}{1 - p} \left [ (1 - p^N) G^{(k)}(0) - p^{N} \sum_{l = 1}^k {k \choose l} N^l G^{(k-l)}(0) \right ].
\end{equation}
\end{proof}
This result provides a representation of the sum $\sum_{n = 1}^{N} n^k p^n$ in terms of the first $k$ moments of the geometric distribution with parameter $1 - p$. Additionally, taking the limit $N \rightarrow \infty$ reduces this expression to the $k^{\text{th}}$ moment of the geometric distribution with parameter $1 - p$ multiplied by $p/(1 - p)$. As the original sum is the $k^{\text{th}}$ geometric moment truncated at $N$ and scaled in the same way, this result is expected.

We can use this theorem by expanding Equation \ref{eq:asymTrace} to give
\begin{eqnarray}
  \frac{1}{M} \frob{\m{R}_{\Sigma_e}}^2 & = & \frac{2}{M^3} \sum_{m = 1}^{M-1} (M-m) m^2 \left ( \rho^{2m} - 2 \rho^{M} + \rho^{2(M - m)} \right ), \nonumber \\
  & & \nonumber \\
  & = & \frac{2}{M^2} \left [ \sum_{m = 1}^{M-1} (M^2 - 2Mm + 2m^2) \rho^{2m} - 2 \rho^M \sum_{m = 1}^{M-1} m^2 \right ] \nonumber \\
  & & - \frac{2}{M^3} \left [ \sum_{m = 1}^{M-1} (M^3 - 3M^2m + 3Mm^2) \rho^{2m} - 2 \rho^M \sum_{m = 1}^{M-1} m^3 \right ], \nonumber \\
  & & \nonumber \\
  & = & 2 \left [ \sum_{m = 1}^{M-1} \left ( \frac{m}{M} - \frac{m^2}{M^2} \right ) \rho^{2m} - 2 \rho^M \left ( \sum_{m = 1}^{M-1} \frac{m^2}{M^2} - \frac{m^3}{M^3} \right ) \right ] \label{eq:frobeniussimp}
\end{eqnarray}
which has terms proportional to Equation \ref{eq:kfirstmoments} for $k = 1$ and $2$. No further evaluation is necessary to see
$$ \lim_{M \rightarrow \infty} \frac{1}{M} \frob{\m{R}_{\Sigma_e}}^2 = 0,$$
as these first two moments are finite for the geometric distribution and the final term is strictly less than $4M \rho^M$, which is zero in the limit. As a result we also get
$$ \lim_{M \rightarrow \infty} \frac{1}{\sqrt{M}} \frob{\m{R}_{\Sigma_e}} = 0$$
and so $\m{C}_{\Sigma_e}$ is asymptotically equivalent to $\sm{\Sigma}_e$.

The case of $\m{C}_{\Sigma}$ and $\sm{\Sigma}$ is less straightforward, as Equation \ref{eq:asymTrace} becomes
\begin{equation} \label{eq:nonexponent}
  \frac{1}{M} \frob{\m{R}_{\Sigma}}^2 = \frac{2}{M^3} \sum_{m = 1}^{M-1} (M-m) m^2 \left ( \rho_m^2 - 2 \rho_{M-m} \rho_m + \rho_{M - m}^2 \right ).
\end{equation}
This expression will be zero in the limit $M \rightarrow \infty$ if the series $\rho_m$ is absolutely convergent and the partial sums $\sum_{m = 1}^{M-1}m^3 \rho_m$ and $\sum_{m = 1}^{M-1} m^2 \rho_m$ are $O(M^2)$ and $O(M)$, respectively. If these conditions are met, then $\m{C}_{\Sigma}$ and $\sm{\Sigma}$ are asymptotically equivalent.

\section{Comparing approximations} \label{c:multipleTesting:rateConverge}
  
$\m{R}_{\Sigma_e}$ can be evaluated further by considering Equation \ref{eq:kfirstmoments} for $k = 1$ and $2$ with $p = \rho^2$ and $N = M-1$:
\begin{eqnarray}
  \sum_{m = 1}^{M-1} m \rho^{2m} & = & \frac{\rho^2}{1 - \rho^2} \Bigg [ \frac{1 - \rho^{2(M-1)}}{1 - \rho^2} - \rho^{2(M-1)} (M - 1) \Bigg ], \label{eq:k1sum} \\
  & & \nonumber \\
  \sum_{m = 1}^{M-1} m^2 \rho^{2m} & = & \frac{\rho^2}{1 - \rho^2} \Bigg [ \frac{(1 - \rho^{2(M-1)})(1 + \rho^2)}{(1 - \rho^2)^2} - \frac{2 \rho^{2(M-1)}}{1 - \rho^2}(M - 1) \nonumber \\
  & & \hspace{1.8cm} - \rho^{2(M-1)} (M-1)^2 \Bigg ]. \label{eq:k2sum}
\end{eqnarray}
Substituting Equations \ref{eq:k1sum} and \ref{eq:k2sum} into Equation \ref{eq:frobeniussimp} gives
\begin{eqnarray}
  \sum_{m = 1}^{M-1} \left ( \frac{m}{M} - \frac{m^2}{M^2} \right ) \rho^{2m} & = & \frac{\rho^2}{M(1 - \rho^2)} \Bigg [ \frac{1 + \rho^{2(M-1)}}{1 - \rho^2} - \frac{1 + \rho^2 + \rho^{2(M-1)} - 3\rho^{2M}}{(1 - \rho^2)^2}\frac{1}{M} \nonumber \\
  & & \hspace{2.2cm} - \rho^{2(M-1)} \left ( 1 - \frac{1}{M} \right )\Bigg ]. \label{eq:mm2rhosum}
\end{eqnarray}
Simplifying
$$2 \rho^M \left ( \sum_{m = 1}^{M-1} \frac{m^2}{M^2} - \frac{m^3}{M^3} \right )$$
comes by substituting the equations for the sums of squares and cubes of the first $M-1$ integers:
\begin{eqnarray}
  2 \rho^M \left ( \sum_{m = 1}^{M-1} \frac{m^2}{M^2} - \frac{m^3}{M^3} \right ) & = & 2 \rho^M \left ( \frac{(M-1)M(2M-1)}{6M^2} - \frac{((M-1)M)^2}{4M^3} \right ) \nonumber \\
  & & \nonumber \\
  & = & 2 \rho^M  \left ( \frac{2M^3 - 3M^2 + M}{6M^2} - \frac{M^4 - 2M^3 + M^2}{4M^3} \right ) \nonumber \\
  & & \nonumber \\
  & = & \frac{1}{6} \rho^M  \left ( M - \frac{1}{M} \right ) \label{eq:diffSqNCube}
\end{eqnarray}
Taking the difference of Equations \ref{eq:mm2rhosum} and \ref{eq:diffSqNCube} gives an expression for the squared Frobenius norm of $\m{R}_{\Sigma_e}$:
\begin{eqnarray}
  \frac{1}{M} \frob{\m{R}_{\Sigma_e}}^2 & = &  \frac{2\rho^2}{M(1 - \rho^2)} \Bigg [ \frac{1 + \rho^{2(M-1)}}{1 - \rho^2} - \frac{1 + \rho^2 + \rho^{2(M-1)} - 3\rho^{2M}}{M(1 - \rho^2)^2} - \rho^{2(M-1)}  \frac{M - 1}{M}\Bigg ] \nonumber \\
  & & - \frac{1}{3} \rho^M \left ( M - \frac{1}{M} \right ) \label{eq:frobNormSigE}
\end{eqnarray}

\cite{gray2006toeplitz} suggests a slightly different approximation. Following \cite{grenanderszego1958}, the sum
$$f(x) = \sum_{k = -\infty}^{\infty} \rho_k e^{ikx} = \sum_{k = -\infty}^{\infty} \rho^{\abs{k}} e^{ikx}$$
is known to be critical to the approximation of $\sm{\Sigma}$. \cite{gray2006toeplitz} considers circulant entries generated using the expression
\begin{equation} \label{eq:multipleTesting:grayEq}
  c_m = \begin{cases}
    \rho_0 & \text{for } m = 0, \\
    & \\
    \frac{1}{M} \sum_{j = 0}^{M-1} f\left ( \frac{2\pi j}{M} \right ) e^{\frac{2 \pi i j m}{M}} & \text{otherwise}, \\
  \end{cases}
\end{equation}
the second case can be expressed
\begin{equation*}
  \begin{aligned}
    \frac{1}{M} \sum_{j = 0}^{M-1} f\left ( \frac{2\pi j}{M} \right ) e^{\frac{2 \pi i j m}{M}} & = \frac{1}{M} \sum_{j = 0}^{M-1} \sum_{k = -\infty}^{\infty} \rho^{\abs{k}} e^{ik \frac{2\pi j}{M}} e^{\frac{2 \pi i j m}{M}} \\
    & \\
    & = \sum_{k = -\infty}^{\infty} \rho^{\abs{k}} \frac{1}{M} \sum_{j = 0}^{M-1} e^{\frac{2 \pi i j}{M}(m+k)} \\
    & \\
    & = \sum_{k = -\infty}^{\infty} \rho^{\abs{k}} \ind{0}{(m+k)\text{mod} M} \\
  \end{aligned}
\end{equation*}
as the second sum is the sum of the squared $M^{\text{th}}$ roots of unity, which are orthonormal. Now
\begin{equation*}
  \begin{aligned}
    \sum_{k = -\infty}^{\infty} \rho^{\abs{k}} \ind{0}{(m+k)\text{mod} M} & = \sum_{k = -\infty}^{\infty} \rho^{\abs{-m + kM}} \\
    & \\
    & = \sum_{k = -\infty}^0 \rho^{m - kM} + \sum_{k = 1}^{\infty} \rho^{- m + kM} \\
    & \\
    & = \rho^m \frac{1}{1 - \rho^M} + \rho^{-m} \frac{\rho^M}{1 - \rho^M} \\
    & \\
    & = \frac{1}{1 - \rho^M} \big ( \rho^m + \rho^{M-m} \big ) \\
  \end{aligned}
\end{equation*}
when $\rho < 1$. Therefore, Equation \ref{eq:multipleTesting:grayEq} becomes
\begin{equation} \label{eq:multipleTesting:grayApprox}
  c_m = \begin{cases}
    \rho_0 & \text{for } m = 0, \\
    & \\
    \frac{1}{1 - \rho^M} \big ( \rho^m + \rho^{M-m} \big )  & \text{otherwise},\\
  \end{cases}
\end{equation}
in the particular case of $\rho_m = \rho^m$. So, while $\m{C}_{\Sigma_e}$ has entries which are a weighted average of $\rho^m$ and $\rho^{M-m}$, this approximation instead takes a sum scaled by $1 - \rho^M$. Define $\m{C}_{GS}$ as the circulant matrix with these entries, and let $\m{R}_{GS}$ be $\sm{\Sigma_e} - \m{C}_{GS}$. Then
\begin{equation} \label{eq:frobNormGren}
    \frac{1}{M} \frob{\m{R}_{GS}}^2 = \frac{2}{(1 - \rho^M)^2} \left [ \frac{\rho^2 (1 - \rho^{2M})^2}{(1 - \rho^2)^2} \frac{1}{M} + (M - 2) \rho^{2M} \right ].
\end{equation}
As $\rho^M = O(1/M^k)$ for $0 \leq \rho < 1$ and $k < \infty$, Equation \ref{eq:frobNormGren} gives
\begin{equation} \label{eq:GrenLeading}
  \frac{1}{\sqrt{M}} \frob{\m{R}_{GS}} = \frac{\sqrt{2} \rho}{(1 - \rho^2) (1 - \rho^M)} \frac{1}{\sqrt{M}} + O\left(\frac{1}{M}\right)
\end{equation}
while the same logic applied to Equation \ref{eq:frobNormSigE} gives
\begin{equation} \label{eq:nearLeading}
  \frac{1}{\sqrt{M}} \frob{\m{R}_{\Sigma_e}} = \frac{\sqrt{2} \rho}{1 - \rho^2} \frac{1}{\sqrt{M}} + O\left(\frac{1}{M}\right)
\end{equation}
The leading terms of which give the relationship for large $M$
\begin{equation*}
 \frac{1}{\sqrt{M}}\frob{\m{R}_{\Sigma_e}} \approx (1 - \rho^M) \frac{1}{\sqrt{M}} \frob{\m{R}_{GS}}.
\end{equation*}

\begin{figure}
  \centering
  \begin{tabular}{cc}
    \includegraphics[scale = 0.45]{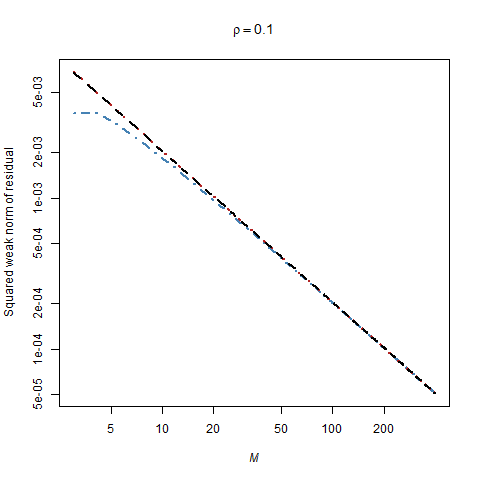} & \includegraphics[scale = 0.45]{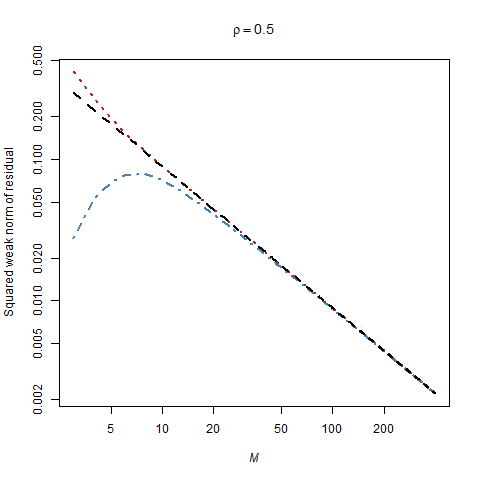} \\
    \includegraphics[scale = 0.45]{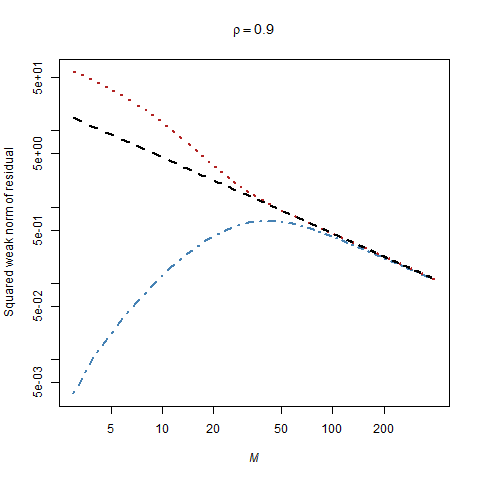} & \includegraphics[scale = 0.45]{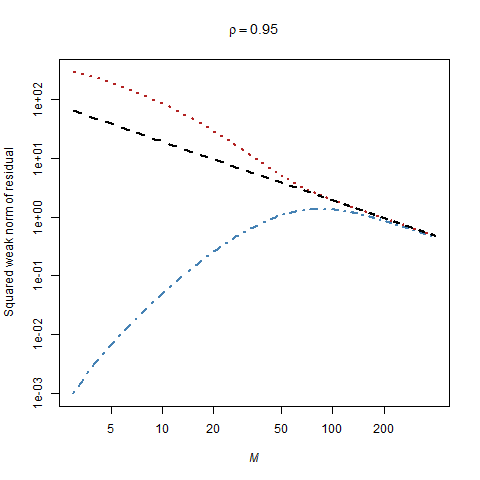}
  \end{tabular}
  \caption{Log-scale plots of the squared Frobenius norms of the residuals $\m{R}_{\Sigma_e}$ and $\m{R}_{GS}$ divided by $M$ for a selection of $\rho$ values over $M = 3$ to 400. The blue dotted line corresponds to $\log \frac{1}{M} \frob{\m{R}_{\Sigma_e}}^2$, the red dot-dashed line to  $\log \frac{1}{M} \frob{\m{R}_{GS}}^2$, and the black dashed line to the log of the leading term given by Equation \ref{eq:nearLeading}} \label{fig:rhoPlots}
\end{figure}
As $1 - \rho^M < 1$, $\m{R}_{\Sigma_e}$ converges to zero faster than $\m{R}_{GS}$ with a greater difference in convergence for $\rho$ values close to 1. A comparison of the values of Equations \ref{eq:frobNormSigE} and \ref{eq:frobNormGren} is shown in Figure \ref{fig:rhoPlots}.

The nearest circulant approximation performs considerably better than the approximation suggested by \cite{gray2006toeplitz}. Using the leading term line as a guide, the nearest circulant matrix is uniformly less than this leading term, while the other is greater than the leading term. Moreover, while the nearest circulant improves in quality for small $M$ as $\rho$ nears 1, the approximation of \cite{gray2006toeplitz} performs poorer relative to the leading term. This makes sense, as for $\rho$ close to one, the absolute difference between $\rho^m$ and $\rho^{(M-m)}$ decreases, making the weighted average entries of $\m{C}_{\Sigma_e}$ closer to the true values. In contrast, the other approximation adds these values and scales them, and so generally has residual terms with larger magnitude. The nearest circulant is orders of magnitude closer in the Frobenius norm for small $M$ in these cases.

\section{Conclusion} \label{sec:conc}

We have derived here a circulant approximation of $\sm{\Sigma}$, $\m{C}_{\Sigma}$, which is optimal in the sense of the Frobenius norm. $\m{C}_{\Sigma}$ is also symmetric, and so is guaranteed to have real eigenvalues. For applications such as multiple test adjustment, this is a highly desirable quality. Additionally, it was proven that only symmetric real circulants are guaranteed to have this property. Beyond multiple test adjustment, any application requiring the eigenvalues of a Toeplitz matrix can use $\m{C}_{\Sigma}$ for its optimality and guaranteed real eigenvalues.

$\m{C}_{\Sigma}$ is also much better than classic approximation provided in \cite{grenanderszego1958} in the case of $\sm{\Sigma}_e$, where $\rho_m = \rho^m$. In particular, as $\rho$ approaches one, $\m{C}_{\Sigma_e}$ performs orders of magnitude better. In the limit of large $M$, both approximations have the same behaviour.

In demonstrating this improved behaviour, another general result was demonstrated. It was proven that the sum $\sum_{n = 1}^{N} n^k p^n$ can be expressed as a linear combination of the first $k$ geometric moments. This expression was provided explicitly for interested readers.

%% BIBLIOGRAPHY
\bibliographystyle{plain}
\bibliography{../Bibliography/fullbib}

\end{document}